\documentclass[11pt]{article}

\usepackage{amssymb,amsmath,amsfonts,amsthm}
\usepackage{latexsym}
\usepackage{graphics}
\usepackage{indentfirst}
\newtheorem{innercustomthm}{Theorem}
\newenvironment{customthm}[1]
  {\renewcommand\theinnercustomthm{#1}\innercustomthm}
  {\endinnercustomthm}

\newtheorem*{thm*}{Theorem}
\newtheorem{thm}{Theorem}

\newtheorem{pro}[thm]{Proposition}

\newtheorem{cor}[thm]{Corollary}

\newtheorem{ques}[thm]{Question}

\newcommand{\N}{\mathbb{N}}
\newcommand{\HH}{\mathcal{H}}

\begin{document}

\title{A Deletion-Contraction Relation for the DP Color Function}

\author{Jeffrey A. Mudrock$^1$}

\footnotetext[1]{Department of Mathematics, College of Lake County, Grayslake, IL 60030.  E-mail:  {\tt {jmudrock@clcillinois.edu}}}

\maketitle

\begin{abstract}

DP-coloring is a generalization of list coloring that was introduced in 2015 by Dvo\v{r}\'{a}k and Postle.  The chromatic polynomial of a graph $G$, denoted $P(G,m)$, is equal to the number of proper $m$-colorings of $G$.  A well-known tool for computing the chromatic polynomial of graph $G$ is the deletion-contraction formula which relates $P(G,m)$ to the chromatic polynomials of two smaller graphs.    The DP color function of a graph $G$, denoted $P_{DP}(G,m)$, is a DP-coloring analogue of the chromatic polynomial, and $P_{DP}(G,m)$ is the minimum number of DP-colorings of $G$ over all possible $m$-fold covers.  In this paper we present a deletion-contraction relation for the DP color function.  To make this possible, we extend the definition of the DP color function to multigraphs.  We also introduce the dual DP color function of a graph $G$, denoted $P^*_{DP}(G,m)$, which counts the maximum number of DP-colorings of $G$ over certain $m$-fold covers.  We show how the dual DP color function along with our deletion-contraction relation yields a new general lower bound on the DP color function of a graph.

\medskip



\end{abstract}

\section{Introduction}\label{intro}

In this paper all graphs are nonempty, finite, undirected loopless multigraphs.  For the purposes of this paper, a simple graph is a multigraph without any parallel edges between vertices.  Generally speaking we follow West~\cite{W01} for terminology and notation.  The set of natural numbers is $\N = \{1,2,3, \ldots \}$.  For $m \in \N$, we write $[m]$ for the set $\{1, \ldots, m \}$.  If $G$ is a graph and $S, U \subseteq V(G)$, we use $G[S]$ for the subgraph of $G$ induced by $S$, and we use $E_G(S, U)$ for the set consisting of all the edges in $E(G)$ that have both endpoints in $S \cup U$, at least one endpoint in $S$, and at least one endpoint in $U$.  Additionally, $N_G(S)$ is the set of all vertices in $V(G)$ that are adjacent in $G$ to at least one vertex in $S$.  When $e \in E(G)$, $G-e$ denotes the graph obtained from $G$ by deleting edge $e$, and $G \cdot e$ denotes the graph obtained from $G$ by contracting the edge $e$ which means that edge $e$ is deleted, the endpoints of $e$ are identified as the same vertex, and any loops formed by identifying the endpoints of $e$ as the same vertex are deleted.  Also, when $E \subseteq E(G)$, $G-E$ denotes the graph obtained from $G$ by deleting each edge in $E$.  When $u,v \in V(G)$ we use $E_G(u,v)$ to denote the set of edges in $E(G)$ with endpoints $u$ and $v$ (note $E_G(u,v) = E_G(v,u)$), and we let $e_G(u,v)$ denote the number of elements in $E_G(u,v)$.  If $u$ has one neighbor in $V(G)$, we say that $u$ is a \emph{pendant vertex}.  When $G$ is a multigraph, the \emph{underlying graph of $G$} is the simple graph formed by deleting all parallel edges of $G$.  When $G$ is a simple graph, we can refer to edges by their endpoints; for example, if $u$ and $v$ are adjacent in the simple graph $G$, $uv$ or $vu$ refers to the edge between $u$ and $v$.

\subsection{List Coloring and DP-Coloring}

In classical vertex coloring we wish to color the vertices of a graph $G$ with up to $m$ colors from $[m]$ so that adjacent vertices receive different colors, a so-called \emph{proper $m$-coloring}.  List coloring is a well-known variation on classical vertex coloring that was introduced independently by Vizing~\cite{V76} and Erd\H{o}s, Rubin, and Taylor~\cite{ET79} in the 1970s.  For list coloring, we associate a \emph{list assignment} $L$ with a graph $G$ such that each vertex $v \in V(G)$ is assigned a list of colors $L(v)$ (we say $L$ is a list assignment for $G$).  Then, $G$ is \emph{$L$-colorable} if there exists a proper coloring $f$ of $G$ such that $f(v) \in L(v)$ for each $v \in V(G)$ (we refer to $f$ as a \emph{proper $L$-coloring} of $G$).  A list assignment $L$ is called an \emph{$m$-assignment} for $G$ if $|L(v)|=m$ for each $v \in V(G)$.  We say $G$ is \emph{$m$-choosable} if $G$ is $L$-colorable whenever $L$ is an $m$-assignment for $G$.  Note that if $G$ is $m$-choosable, then a proper $m$-coloring for $G$ exists since the list assignment for $G$ that assigns $[m]$ to each element in $V(G)$ is an $m$-assignment for $G$.  

In 2015, Dvo\v{r}\'{a}k and Postle~\cite{DP15} introduced a generalization of list coloring called DP-coloring (they called it correspondence coloring) in order to prove that every planar graph without cycles of lengths 4 to 8 is 3-choosable. DP-coloring has been extensively studied over the past 6 years (see e.g.,~\cite{B17, BK182, KM20, KO18, LL19, LLYY19, Mo18, M18}). Intuitively, DP-coloring is a variation on list coloring where each vertex in the graph still gets a list of colors, but identification of which colors are different can change from edge to edge.  Due to this property, DP-coloring multigraphs is not as simple as coloring the corresponding underlying graph (see~\cite{BKP17}).  We now give the formal definition.  Suppose $G$ is a multigraph.  A \emph{cover} of $G$ is a triple $\mathcal{H} = (L,H,M)$ where $L$ is a function that assigns to each element of $V(G)$ a nonempty finite set, $H$ is a multigraph with vertex set $\bigcup_{v \in V(G)} L(v)$, and $M$ is a function that assigns to each $e \in E(G)$ a matching $M(e)$ with the property that each edge in $M(e)$ has one endpoint in $L(u)$ and one endpoint in $L(v)$ where $u$ and $v$ are the endpoints of $e$.  Moreover, $L$, $H$, and $M$ satisfy the following conditions~\footnote{When we construct a cover, we will often omit the definition of $M$ when $G$ is simple since its definition will be obvious in such cases.}: 

\vspace{5mm}

\noindent (1) For distinct vertices $u,v \in V(G)$, $L(u) \cap L(v) = \emptyset$; \\
(2) For every $u \in V(G)$, $H[L(u)]=K_{|L(u)|}$; \\
(3)  For distinct edges $e_1, e_2 \in E(G)$, $M(e_1) \cap M(e_2) = \emptyset$; \\
(4)  For distinct vertices $u, v \in V(G)$, the set of edges between $L(u)$ and $L(v)$ in $H$ is $\bigcup_{e \in E_G(u,v)} M(e)$. 

\vspace{5mm}

Note that by conditions (3) and (4) in the above definition $H$ may contain parallel edges.  Furthermore, note that if $G$ is a simple graph, $H$ must be simple.  

Suppose $\mathcal{H} = (L,H,M)$ is a cover of $G$.  An \emph{$\mathcal{H}$-coloring} of $G$ is an independent set in $H$ of size $|V(G)|$.  It is immediately clear that an independent set $I \subseteq V(H)$ is an $\mathcal{H}$-coloring of $G$ if and only if $|I \cap L(u)|=1$ for each $u \in V(G)$.  We say $\mathcal{H}$ is \emph{$m$-fold} if $|L(u)|=m$ for each $u \in V(G)$.  Moreover, we say that $\mathcal{H}$ is a \emph{full $m$-fold cover} of $G$ if $|E_H(L(u),L(v))|= e_G(u,v) m$ whenever $u$ and $v$ are distinct vertices of $G$.  It is worth noting that given an $m$-assignment $L$ for a graph $G$, it is easy to construct an $m$-fold cover $\mathcal{H}'$ of $G$ such that $G$ has an $\mathcal{H}'$-coloring if and only if $G$ has a proper $L$-coloring.   

Suppose $\mathcal{H} = (L,H,M)$ is an $m$-fold cover of a multigraph $G$.  Suppose that $H'$ is the underlying graph of $H$.  We say that $\mathcal{H}$ has a \emph{canonical labeling} if it is possible to name~\footnote{When $\mathcal{H}=(L,H,M)$ has a canonical labeling, we will always refer to the vertices of $H$ using this naming scheme.} the vertices of $H$ as $L(u) = \{ (u,j) : j \in [m] \}$ for each $u \in V(G)$ so that whenever $e_G(u,v) \geq 1$, $(u,j)(v,j) \in E(H')$ for each $j \in [m]$ and $|E_{H'}(L(u),L(v))|=m$.  Suppose $\mathcal{H}$ has a canonical labeling and $G$ has a proper $m$-coloring.  Then, if $\mathcal{I}$ is the set of $\mathcal{H}$-colorings of $G$ and $\mathcal{C}$ is the set of proper $m$-colorings of $G$, the function $f: \mathcal{C} \rightarrow \mathcal{I}$ given by $f(c) = \{ (v, c(v)) : v \in V(G) \}$ is a bijection.  So, finding an $\HH$-coloring of $G$ is equivalent to finding a proper $m$-coloring of $G$. 
 
\subsection{Counting Proper Colorings, List Colorings, and DP-Colorings}

In 1912 Birkhoff introduced the notion of the chromatic polynomial in hopes of using it to make progress on the four color problem.  For $m \in \N$, the \emph{chromatic polynomial} of a graph $G$, $P(G,m)$, is the number of proper $m$-colorings of $G$.  It can be shown that $P(G,m)$ is a polynomial in $m$ of degree $|V(G)|$ (see~\cite{B12}).  For example, $P(K_n,m) = \prod_{i=0}^{n-1} (m-i)$ and $P(T,m) = m(m-1)^{n-1}$ whenever $T$ is a tree on $n$ vertices (see~\cite{W01}).  Importantly, in 1940 the deletion-contraction formula for chromatic polynomials~\footnote{From this point forward, for the sake of brevity, we will refer to this as the Deletion-Contraction Formula.} was introduced which is an important tool for recursively finding the chromatic polynomial of a graph.

\begin{thm} [Deletion-Contraction Formula~\cite{BS40}] \label{thm: chromaticpoly}
If $G$ is a simple graph and $e \in E(G)$, then for each $m \in \N$, $P(G,m) = P(G-e,m) - P(G \cdot e, m)$. 
\end{thm}

One easy application of the Deletion-Contraction Formula is that it can be used along with induction and the chromatic polynomial formulas for trees and $K_3$ to prove that $P(C_n,m) = (m-1)^n + (-1)^n (m-1)$ whenever $n \geq 3$.  It is worth noting that this formula for $P(C_n,m)$ still works when $n=2$ since the proper $m$-colorings of $C_2$ directly correspond to the proper $m$-colorings of $K_2$.

The notion of chromatic polynomial was extended to list coloring in the 1990s~\cite{AS90}.   In particular, if $L$ is a list assignment for $G$, we use $P(G,L)$ to denote the number of proper $L$-colorings of $G$. The \emph{list color function} $P_\ell(G,m)$ is the minimum value of $P(G,L)$ where the minimum is taken over all $m$-assignments $L$ for $G$.  It is clear that $P_\ell(G,m) \leq P(G,m)$ for each $m \in \N$ since we must consider the $m$-assignment that assigns $[m]$ to each vertex of $G$ when considering all possible $m$-assignments for $G$.  In general, the list color function can differ significantly from the chromatic polynomial for small values of $m$.  However, for large values of $m$, Wang, Qian, and Yan~\cite{WQ17} (improving upon results in~\cite{D92} and~\cite{T09}) showed that if $G$ is a connected graph with $l$ edges, then $P_{\ell}(G,m)=P(G,m)$ whenever $m > (l-1)/\ln(1+ \sqrt{2})$.

Notice that Wang, Qian, and Yan's result implies that the Deletion-Contraction Formula holds when the chromatic polynomial is replaced with the list color function and $m$ is sufficiently large.  However, for small $m \in \N$ there exists $G$ such that $P_\ell(G,m) \neq P_\ell(G-e,m) - P_\ell(G \cdot e, m)$.  For example, if $G = K_{2,4}$ and $e \in E(G)$, then it is easy to see $P_\ell(G,2)=0$.  However, $P_\ell(G-e,m) - P_\ell(G \cdot e, m) = 2-0 = 2$.

In 2019, Kaul and Mudrock introduced a DP-coloring analogue of the chromatic polynomial called the DP color function in hopes of gaining a better understanding of DP-coloring and using it as a tool for making progress on some open questions related to the list color function~\cite{KM19}.  Since its introduction in 2019, the DP color function has received some attention in the literature (see e.g.,~\cite{BH21, DY21, HK21, KM21, MT20}). 

The motivation for this paper is to find an analogue of the Deletion-Contraction Formula for the DP color function.  In order to do this, we need to extend the original definition of the DP color function to multigraphs.  Specifically, suppose $\mathcal{H} = (L,H,M)$ is a cover of a multigraph $G$.  Let $P_{DP}(G, \mathcal{H})$ be the number of $\mathcal{H}$-colorings of $G$.  Then, the \emph{DP color function} of $G$, $P_{DP}(G,m)$, is the minimum value of $P_{DP}(G, \mathcal{H})$ where the minimum is taken over all full $m$-fold covers $\mathcal{H}$ of $G$.~\footnote{We take $\N$ to be the domain of the DP color function of any multigraph.}  It is easy to show that for any $m \in \N$, $P_{DP}(G, m) \leq P_\ell(G,m) \leq P(G,m)$.  Note that if $G$ is a disconnected graph with components: $H_1, H_2, \ldots, H_t$, then $P_{DP}(G, m) = \prod_{i=1}^t P_{DP}(H_i,m)$.  So, we will assume that any graph $G$ introduced from this point forward is connected unless otherwise noted.

Unlike the list color function, it is well-known that $P_{DP}(G,m)$ does not necessarily equal $P(G,m)$ for sufficiently large $m$.  Indeed Dong and Yang~\cite{DY21} recently generalized a result of Kaul and Mudrock~\cite{KM19} and showed that if $G$ is a simple graph that contains an edge $e$ such that the length of a shortest cycle containing $e$ is even, then there exists an $N \in \N$ such that $P_{DP}(G,m) < P(G,m)$ whenever $m \geq N$.  A related result that will be important for this paper was recently proven by Mudrock and Thomason.

\begin{thm} [\cite{MT20}] \label{thm: general}
Suppose $g$ is an odd integer with $g \geq 3$.  If $G$ is a simple graph on $n$ vertices with girth at least $g$, then $P(G,m) - P_{DP}(G,m) = O(m^{n-g})$ as $m \rightarrow \infty$.   
\end{thm}

Given a multigraph $G$, the deletion-contraction relation that we prove in this paper will naturally require us to study the full covers of $G$ for which there are as many colorings as possible.  Consequently, we define the \emph{dual DP color function} of $G$, $P^*_{DP}(G,m)$, as the maximum value of $P_{DP}(G, \mathcal{H})$ where the maximum is taken over all full $m$-fold covers $\mathcal{H}$ of $G$.  Clearly, $P_{DP}(G, m) \leq P_\ell(G,m) \leq P(G,m) \leq P^*_{DP}(G,m)$.  Since the dual DP color function is an upper bound on the chromatic polynomial it may be of independent interest. 

\subsection{Outline of Results and Open Questions}

We now present an outline of the paper while also mentioning two open questions.  In Section~\ref{basic} we present some basic results on the DP color function and dual DP color function that will be used later in the paper.  Many of the basic results presented are new since this paper is the first to extend the definition of the DP color function to multigraphs, and this paper introduces the notion of the dual DP color function.  Then in Section~\ref{main}, if $\HH = (L,H,M)$ is a full $m$-fold cover of $G$ and $e \in E(G)$, we begin by giving technical definitions of the covers $\HH - e$ and $\HH \cdot e$ which are covers of $G-e$ and $G \cdot e$ respectively that are obtained from $\HH$.  After introducing these definitions, we prove our deletion-contraction relation for the DP color function which we now state.

\begin{thm} \label{thm: deletioncontraction}
Suppose $G$ is a multigraph, $\HH = (L,H,M)$ is a full $m$-fold cover of $G$, and $e$ is an edge in $G$ with endpoints $u$ and $v$.  Then,  
$$P_{DP}(G, \mathcal{H}) \geq P_{DP}(G-e, \mathcal{H} - e) - P_{DP}(G \cdot e, \mathcal{H} \cdot e).$$
Consequently,
$$P_{DP}(G-e,m) - P^*_{DP}(G \cdot e, m ) \leq P_{DP}(G,m).$$
Moreover, if for each $f \in E_G(u,v) - \{e\}$, no edge in $M(f)$ has the same endpoints as any edge in $M(e)$, then
$$P_{DP}(G, \mathcal{H}) = P_{DP}(G-e, \mathcal{H} - e) - P_{DP}(G \cdot e, \mathcal{H} \cdot e).$$
Consequently, when $e_G(u,v) = 1$, 
$$ P^*_{DP}(G,m) \leq P^*_{DP}(G-e,m) - P_{DP}(G \cdot e, m ).$$
\end{thm}

After we prove Theorem~\ref{thm: deletioncontraction}, we show how it can be used to help determine the DP color function and dual DP color function of $C_n$ for each $n \geq 3$ and simple unicyclic graphs (i.e., connected simple graphs containing exactly one cycle).  In fact, the lower bound on the DP color function and upper bound on the dual DP color function provided by Theorem~\ref{thm: deletioncontraction} are tight for simple unicyclic graphs.  We will also see that these bounds are tight for trees as well.  We end Section~\ref{main} by showing examples which demonstrate that the bounds in Theorem~\ref{thm: deletioncontraction} are tight for graphs that are neither trees nor simple unicyclic graphs.  This leads us to pose the following two open questions.

\begin{ques} \label{ques: lower}
For which multigraphs $G$ does there exist an $e \in E(G)$ such that $P_{DP}(G-e,m) - P^*_{DP}(G \cdot e, m ) = P_{DP}(G,m)$ for infinitely many $m \in \N$?
\end{ques}

\begin{ques} \label{ques: upper}
For which multigraphs $G$ does there exist an edge $e$ with endpoints $u$ and $v$ such that $e_G(u,v) = 1$ and $ P^*_{DP}(G,m) = P^*_{DP}(G-e,m) - P_{DP}(G \cdot e, m )$ for infinitely many $m \in \N$?
\end{ques}

Finally, in Section~\ref{dual} we present a nontrivial application of Theorem~\ref{thm: deletioncontraction}.  In particular, we show how Theorems~\ref{thm: general} and~\ref{thm: deletioncontraction} can be used to prove the following result on the asymptotics of the dual DP color function.

\begin{thm} \label{thm: dual}
Suppose $g$ is an odd integer with $g \geq 3$.  If $G$ is a simple graph on $n$ vertices with girth at least $g-1$, then $P^*_{DP}(G, m) - P(G,m) = O(m^{n-g+1})$ as $m \rightarrow \infty$. 
\end{thm}

Theorem~\ref{thm: dual} is best possible since we will see below that for each $k \in \N$,  $P^*_{DP}(C_{2k+1}, m) - P(C_{2k+1},m) = \Theta(m)$ as $m \rightarrow \infty$. 

\section{Basic Results} \label{basic}

In this section we prove some basic results about the DP color function and dual DP color function that will be needed in the final two sections of the paper. Our first basic result has already appeared in the literature, and we will use it frequently.

\begin{pro} [\cite{KM19}] \label{pro: tree}
Suppose $T$ is a tree and $\mathcal{H} = (L, H, M)$ is a full $m$-fold cover of $T$.  Then, $\mathcal{H}$ has a canonical labeling.
\end{pro}

As was mentioned in Section~\ref{intro}, if $\HH$ is a full $m$-fold cover of $G$ with a canonical labeling, then finding an $\HH$-coloring of $G$ is equivalent to finding a proper $m$-coloring of $G$.  So, our next result follows immediately from Proposition~\ref{pro: tree}.

\begin{cor} \label{cor: tree}
If $T$ is a tree on $n$ vertices and $m \in \N$, then $P_{DP}(T,m) = P_{DP}^*(T,m) = P(T,m) = m(m-1)^{n-1}$.
\end{cor} 

Next, we prove a basic fact on the relationship between the DP color function (resp. dual DP color function) of a multigraph $G$ and the DP color function (resp. dual DP color function) of its underlying graph.  Importantly, this fact will allow us to restrict our attention to simple graphs when working with the dual DP color function.

\begin{pro} \label{pro: underlying}
Suppose that $G$ is a multigraph and $U$ is the underlying graph of $G$.  Then, $P_{DP}(G,m) \leq P_{DP}(U,m)$ and $P^*_{DP}(G,m) = P^*_{DP}(U,m)$ for each $m \in \N$.  
\end{pro}

\begin{proof}
Since $U$ is a spanning subgraph of $G$ it is immediately clear that if $\HH = (L,H,M)$ is a full $m$-fold cover of $G$, then there is a full $m$-fold cover $\HH' = (L,H',M')$ of $U$ such that $H'$ is a subgraph of $H$.  Consequently, $P^*_{DP}(G,m) \leq P^*_{DP}(U,m)$.  Conversely, if $\HH' = (L,H',M')$ is a full $m$-fold cover of $U$, then there is a full $m$-fold cover $\HH = (L,H,M)$ of $G$ such that $H'$ is a subgraph of $H$.  Consequently, $P_{DP}(G,m) \leq P_{DP}(U,m)$. 

Now, suppose $\HH= (L,H,M)$ is a full $m$-fold cover of $U$ such that $P_{DP}(U,\HH) = P^*_{DP}(U,m)$.  Suppose that $L(x) = \{(x,j): j \in [m]\}$ for each $x \in V(U)$.  We will now construct a full $m$-fold cover $\HH' = (L,H',M')$ of $G$.  Consider each $uv \in E(U)$.  Suppose $M(uv) = \{e_1, \ldots, e_m \}$. For each $e \in E_G(u,v)$ let $M'(e) = \{d_{e,1}, \ldots, d_{e,m} \}$ so that $d_{e,i}$ has the same endpoints as $e_i$ for each $i \in [m]$ and $M'(e) \cap M'(f) = \emptyset$ whenever $e$ and $f$ are distinct elements of $E_G(u,v)$.  Let $H'$ be the graph with vertex set $\bigcup_{x \in V(G)} L(x)$, and edges drawn so that $H'[L(x)]=K_{m}$ for each $x \in V(G)$ and $\bigcup_{e \in E(G)} M'(e) \subseteq E(H')$.  Notice $H$ is the underlying graph of $H'$.  Since $H$ is the underlying graph of $H'$,
$$P^*_{DP}(G,m) \geq P_{DP}(G,\HH') = P_{DP}(U,\HH) = P^*_{DP}(U,m).$$
Thus, $P^*_{DP}(G,m) = P^*_{DP}(U,m)$.   
\end{proof}

We now give an extension of Proposition~20 in~\cite{KM19} to DP color functions of multigraphs.  This result will be particularly useful in Section~\ref{dual} when we use our deletion-contraction relation to study the asymptotics of the dual DP color function.

\begin{pro} \label{pro: order}
Suppose $G$ is a multigraph with at least one edge and $v_1, \ldots, v_n$ is an ordering of the elements of $V(G)$ such that there are precisely $d_i$ edges with $v_i$ as an endpoint and some vertex preceding $v_i$ in the ordering as the other endpoint for each $i \in [n]$. If $D = \max_{i \in [n]} d_i $, then
$$P_{DP}(G,m) \geq \prod_{i=1}^n (m-d_i)$$
whenever $m \geq D$.
\end{pro}

\begin{proof}
First, note that the inequality clearly holds when $m=D$ since $0 \leq P_{DP}(G,D)$.  Suppose that $\mathcal{H}=(L,H,M)$ is an arbitrary full $m$-fold cover of $G$ and $m > D$.  Consider constructing an $\HH$-coloring, $I$, of $G$ via the following inductive procedure.  First, arbitrarily select a vertex $a_1 \in L(v_1)$ and place it in $I$.  Then, for each $i \in \{2, \ldots, n\}$ select a vertex $a_i \in L(v_i)-N_H(\{a_1, \ldots, a_{i-1}\})$ and place it in $I$ (we will justify why this is possible below).  After all $n$ steps of the procedure $I$ is a clearly an $\HH$-coloring of $G$.  

Notice there are clearly $m$ ways to complete the first step of the procedure.  Now, consider the $i^{th}$ step of the procedure where $i \in \{2, \ldots, n\}$.  For each $j \in [i-1]$ we know that in $H$ there are at most $e_G(v_j,v_i)$ neighbors of $a_j$ in $L(v_i)$.  Consequently, 
$$|L(v_i)-N_H(\{a_1, \ldots, a_{i-1}\})| \geq m - \sum_{j=1}^{i-1} e_G(v_j,v_i) = m-d_i.$$ 
So, there are at least $(m - d_i)$ possible choices for $a_i$.  Consequently $\prod_{i=1}^n (m-d_i) \leq P_{DP}(G, \HH)$ which immediately implies $\prod_{i=1}^n (m-d_i) \leq P_{DP}(G, m)$.  
\end{proof} 

Proposition~\ref{pro: order} allows us to give a formula for the DP color function of any graph with the property that its underlying graph is a tree.  Notice that Corollary~\ref{cor: tree} and Proposition~\ref{pro: underlying} already tell us the formula for the dual DP color function of such graphs.

\begin{pro} \label{pro: fattree}
Suppose $G$ is a graph on at least two vertices with the property that its underlying graph is a tree. Suppose $v_1, \ldots, v_n$ is an ordering of the elements of $V(G)$ so that each vertex in the ordering has at most one neighbor preceding it in the ordering.  For each $j \in [n]$ let $d_j$ be the number of edges in $E(G)$ that have $v_j$ as an endpoint and some vertex preceding $v_j$ in the ordering as the other endpoint.  Let $D = \max_{j \in [n]} d_j$.  Then $P_{DP}(G,m) = \prod_{i=1}^n (m-d_i)$ for each $m \geq D$.
\end{pro}

\begin{proof}
Proposition~\ref{pro: order} implies that $P_{DP}(G,m) \geq \prod_{i=1}^n (m-d_i)$ whenever $m \geq D$. Suppose that $m \geq D$.  To complete the proof, we must construct an $m$-fold cover, $\HH = (L,H,M)$, of $G$ such that $P_{DP}(G,\HH) = \prod_{i=1}^n (m-d_i)$. Begin by letting $L(v_i) = \{(v_i,j): j \in [m] \}$ for each $i \in [n]$.  Whenever $v_q$ and $v_r$ are adjacent in $G$ and $E_G(v_q,v_r) = \{e_1, \ldots, e_k \}$, for each $i \in [k]$, let $M(e_i) = \{d_{i,1}, \ldots, d_{i,m} \}$ where $d_{i,j}$ is an edge with endpoints $(v_q,j)$ and $(v_r, ((j+i) \text{ mod } m) + 1)$ for each $j \in [m]$.  Let $H$ be the graph with vertex set $\bigcup_{v \in V(G)} L(v)$ and edges constructed as follows.  Construct edges so that $H[L(v_i)]$ is a complete graph for each $i \in [n]$, and $\bigcup_{e \in E(G)} M(e) \subseteq E(H)$.  

In the case $m = D$ it is clear that $P_{DP}(G, \HH)=0$ since for vertices $u,v \in V(G)$ with $e_G(u,v)=D$, each vertex in $L(u)$ is adjacent in $H$ to each vertex $L(v)$ by construction.  So, $H$ can't contain an independent set of size $|V(G)|$.

Now, suppose $m > D$ and consider constructing an $\HH$-coloring of $G$ by following the $n$-step inductive process described in the proof of Proposition~\ref{pro: order}.  Clearly, each $\HH$-coloring of $G$ can be constructed in exactly one way via this procedure.  Notice that by construction, in the $i^{th}$ step of the procedure, $|L(v_i)-N_{H}(\{a_1, \ldots, a_{i-1}\})| = m - d_i$.  The reason for this is clear when $d_i = 0$, and when $d_i > 0$ this follows from the fact that $v_i$ is adjacent in $G$ to exactly one vertex $v_l$ with $l < i$.  So, of all the vertices in the set $\{a_1, \ldots, a_{i-1}\}$ the vertices in $L(v_i)$ can only be adjacent in $H$ to $a_l$, and by construction $a_l$ is adjacent in $H$ to exactly $d_i$ vertices in $L(v_i)$.  This means that there are exactly $m - d_i$ ways to complete the $i^{th}$ step of the inductive procedure.  It immediately follows that $P_{DP}(G,\HH) = \prod_{i=1}^n (m-d_i)$.    
\end{proof}

Finally, we show that adding a pendant edge to a graph has a predictable effect on the DP color function and dual DP color function of the graph.  

\begin{pro} \label{pro: pendant}
Suppose $G$ is a multigraph and $v$ is a pendant vertex in $G$ of degree $k$ with $k \in \N$.  Suppose $G' = G-v$.  Then, for each $m \geq k$, $P_{DP}(G,m) = (m-k)P_{DP}(G',m)$ and $P^*_{DP}(G,m) = (m-1)P^*_{DP}(G',m)$
\end{pro}

\begin{proof}
 Suppose $u$ is the only neighbor of $v$ in $G$ and $\HH = (L,H,M)$ is an arbitrary full $m$-fold cover of $G$.  Let $L'$ be the restriction of $L$ to $V(G')$, $H' = H - L(v)$, and $M'$ be the restriction of $M$ to $E(G')$.  Then, $\HH' = (L',H',M')$ is a full $m$-fold cover of $G'$.  Notice that an $\HH$-coloring of $G$ can be constructed via the following two step procedure.  First, find an $\HH'$-coloring of $G'$ called $I'$.  Second, choose an element (if such an element exists) in $L(v)$ that is not adjacent in $H$ to the vertex in $I' \cap L'(u)$, and place it in $I'$.  It is clear that if both steps of this procedure can be completed, the result is an $\HH$-coloring of $G$.  Furthermore, the first step can be completed in at least $P_{DP}(G',m)$ ways and at most $P^*_{DP}(G',m)$ ways, and the second step can be completed in at least $(m-k)$ ways and at most $(m-1)$ ways (regardless of how the first step is completed).  So, $(m-k)P_{DP}(G',m) \leq P_{DP}(G, \HH) \leq (m-1)P^*_{DP}(G',m)$ which implies $(m-k)P_{DP}(G',m) \leq P_{DP}(G,m) \leq P_{DP}^*(G,m) \leq (m-1)P^*_{DP}(G',m)$.

Now, suppose that $\HH_1 = (L_1, H_1, M_1)$ is a full $m$-fold cover of $G'$ such that $P_{DP}(G', \HH_1) = P_{DP}(G',m)$.  Suppose that $L_1(x) = \{(x,j) : j \in [m] \}$ for each $x \in V(G')$.  Suppose that $L_2$ is the function on $V(G)$ that agrees with $L_1$ on $V(G')$ and maps $v$ to $\{(v,j) : j \in [m] \}$.  Let $M_2$ be the function on $E(G)$ that agrees with $M_1$ on $E(G')$, and deals with the elements in $E_G(u,v)$ in the following manner.  Suppose $E_G(u,v) = \{e_1, \ldots, e_k \}$; for each $i \in [k]$, let $M_2(e_i) = \{d_{i,1}, \ldots, d_{i,m} \}$ where $d_{i,j}$ is an edge with endpoints $(u,j)$ and $(v, ((j+i) \text{ mod } m) + 1)$ for each $j \in [m]$.  Let $H_2$ be the graph with vertex set $\bigcup_{v \in V(G)} L_2(v)$ and edges constructed as follows.  Construct edges so that $H_2[L_2(x)]$ is a complete graph for each $x \in V(G)$ and $\bigcup_{e \in E(G)} M_2(e) \subseteq E(H_2)$.     

Now, $\HH_2 = (L_2, H_2, M_2)$ is a full $m$-fold cover of $G$.  Moreover, notice that every $\HH_2$-coloring of $G$ can be constructed in exactly one way via the two step procedure described at the beginning of this proof.  Also by construction, we see that the first step of the procedure can be completed in $P_{DP}(G',m)$ ways, and the second step of the procedure can be completed in $(m-k)$ ways (regardless of how the first step is completed).   It follows that $P_{DP}(G,m) \leq P_{DP}(G, \HH_2) = (m-k)P_{DP}(G',m)$ which means that $P_{DP}(G,m) = (m-k)P_{DP}(G',m)$.  A similar construction can be used to show that $P^*_{DP}(G,m) = (m-1)P^*_{DP}(G',m)$.  
\end{proof}

\section{Deletion-Contraction Relation} \label{main}

In this Section we wish to prove Theorem~\ref{thm: deletioncontraction}.  In order to do this, we need two definitions.  Suppose $G$ is a multigraph with $e \in E_G(u,v)$.  Suppose that $G \cdot e$ is the graph obtained from $G$ by contracting the edge $e$ to the vertex $w$.  Notice that each edge in $G \cdot e$ clearly corresponds to exactly one edge in $E(G) - E_G(u,v)$.  Suppose $\HH = (L,H,M)$ is a full $m$-fold cover of $G$ with $L(x) = \{(x,j): j \in [m]\}$ for each $x \in V(G)$.~\footnote{From this point forward, we will assume that the vertices of an $m$-fold cover are named like this unless otherwise noted.}  

Let $\mathcal{H} - e$ be the full $m$-fold cover for $G-\{e\}$ given by $(L,H',M')$ where $H' = H - M(e)$ and $M'$ is the function obtained by restricting the domain of $M$ to $E(G) - \{e\}$.  Now, let $L'$ be the function on $V(G \cdot e)$ that agrees with $L$ on $V(G) - \{u,v \}$ and maps $w$ to $\{(w,j) : j \in [m]\}$.  Since $\HH$ is a full $m$-fold cover, we know $M(e)$ is a perfect matching between $L(u)$ and $L(v)$.  Suppose that $H''$ is the graph obtained from $H$ by contracting the edge in $M(e)$ with $(u,j)$ as an endpoint to $(w,j)$ for each $j \in [m]$, and then subsequently deleting edges so that $H''[L(w)] = K_m$.  Finally, suppose $M''$ is the function on $E(G \cdot e)$ that maps each $f \in E(G \cdot e)$ that does not have $w$ as an endpoint to $M(f)$.  If $f \in E(G \cdot e)$ has $w$ as an endpoint, suppose it corresponds to the edge $f' \in E(G)$ with~\footnote{Notice that $f' \notin E_G(u,v)$.} $u$ (resp. $v$) as an endpoint.  Then $M''(f)$ is obtained from $M(f')$ by changing the edge in $M(f')$ with $(u,j)$ as an endpoint so that its endpoint $(u,j)$ is replaced with $(w,j)$ for each $j \in [m]$ (resp. $M''(f)$ is obtained from $M(f')$ by changing the edge in $M(f')$ with $(v,j)$ as an endpoint so that its endpoint $(v,j)$ is replaced with $(w,j')$ where $(u,j')$ is the other endpoint of the edge in $M(e)$ with $(v,j)$ as an endpoint for each $j \in [m]$).  Then, $\mathcal{H} \cdot e$ is the full $m$-fold cover for $G$ given by $(L',H'',M'')$. 

We are now ready to prove Theorem~\ref{thm: deletioncontraction} which we restate.

\begin{customthm} {\ref{thm: deletioncontraction}}
Suppose $G$ is a multigraph, $\HH = (L,H,M)$ is a full $m$-fold cover of $G$, and $e$ is an edge in $G$ with endpoints $u$ and $v$.  Then,  
$$P_{DP}(G, \mathcal{H}) \geq P_{DP}(G-e, \mathcal{H} - e) - P_{DP}(G \cdot e, \mathcal{H} \cdot e).$$
Consequently,
$$P_{DP}(G-e,m) - P^*_{DP}(G \cdot e, m ) \leq P_{DP}(G,m).$$
Moreover, if for each $f \in E_G(u,v) - \{e\}$, no edge in $M(f)$ has the same endpoints as any edge in $M(e)$, then
$$P_{DP}(G, \mathcal{H}) = P_{DP}(G-e, \mathcal{H} - e) - P_{DP}(G \cdot e, \mathcal{H} \cdot e).$$
Consequently, when $e_G(u,v) = 1$, 
$$ P^*_{DP}(G,m) \leq P^*_{DP}(G-e,m) - P_{DP}(G \cdot e, m ).$$ 
\end{customthm} 

\begin{proof}
We may assume without loss of generality that the edge in $M(e)$ with $(u,j) \in L(u)$ as an endpoint, has $(v, j)$ as its other endpoint for each $j \in [m]$.  Let $\mathcal{I}$ be the set of all $(\mathcal{H} - e)$-colorings of $G - e$.  Let $\mathcal{I}_j$ be the set of all $(\mathcal{H} - e)$-colorings of $G - e$ that contain $(u,j)$ and $(v,j)$.  Clearly, $\mathcal{I} - \bigcup_{j \in [m]} \mathcal{I}_j$ is the set of all $\mathcal{H}$-colorings of $G$, and the sets $\mathcal{I}_1, \ldots, \mathcal{I}_m$ are pairwise disjoint.

Suppose $w$ is the vertex in $G \cdot e$ obtained from contracting $e$.  Now, let $\mathcal{C}_j$ be the set of all $(\mathcal{H} \cdot e)$-colorings of $G \cdot e$ that contain $(w,j)$.  Note that $ \bigcup_{j \in [m]} \mathcal{C}_j$ is the set of all $(\mathcal{H} \cdot e)$-colorings of $G \cdot e$, and the sets $\mathcal{C}_1 ,\ldots, \mathcal{C}_m$ are pairwise disjoint.  Now, for each $j \in [m]$ such that $\mathcal{I}_j \neq \emptyset$, suppose $B_j : \mathcal{I}_j \rightarrow \mathcal{C}_j$ is given by $B_j(I) = (I - \{(u,j) , (v,j) \}) \cup \{(w,j)\}$.  It is easy to check that $B_j$ is a bijection.  So, $|\mathcal{I}_j| \leq |\mathcal{C}_j|$ for each $j \in [m]$ (with equality holding when $\mathcal{I}_j \neq \emptyset$).  Thus,
\begin{align*}
P_{DP}(G, \mathcal{H}) = \left|\mathcal{I} - \bigcup_{j \in [m]} \mathcal{I}_j \right| = |\mathcal{I}| - \sum_{j \in [m]} |\mathcal{I}_j| &\geq |\mathcal{I}| - \sum_{j \in [m]} |\mathcal{C}_j| \\
 &=|\mathcal{I}| - \left| \bigcup_{j \in [m]} \mathcal{C}_j \right| \\
 &= P_{DP}(G-e, \mathcal{H} - e) - P_{DP}(G \cdot e, \mathcal{H} \cdot e)
\end{align*}
as desired.

Now, suppose that for each $f \in E_G(u,v) - \{e\}$, no edge in $M(f)$ has the same endpoints as any edge in $M(e)$.  We claim that $|\mathcal{I}_j| = |\mathcal{C}_j|$ for each $j \in [m]$.  We already know that this equality holds when $\mathcal{I}_j \neq \emptyset$.  So, suppose that for some $l \in [m]$, $\mathcal{I}_l = \emptyset$.  We will show that $\mathcal{C}_l = \emptyset$.  For the sake of contradiction suppose that $A \in \mathcal{C}_l$.  Let $B = (A - \{(w,l) \}) \cup \{(u,l) , (v,l) \}$.  Since  for each $f \in E_G(u,v) - \{e\}$, no edge in $M(f)$ has the same endpoints as any edge in $M(e)$, we know that $(u,l)$ and $(v,l)$ are not adjacent in the graph corresponding to the second coordinate of the cover $\HH - e$.  This means $B \in \mathcal{I}_l$ which is a contradiction.  So, $\mathcal{C}_l = \emptyset$, and we may conclude $|\mathcal{I}_j| = |\mathcal{C}_j|$ for each $j \in [m]$.  This means that the inequality in the above computation becomes equality, and we obtain $P_{DP}(G, \mathcal{H}) = P_{DP}(G-e, \mathcal{H} - e) - P_{DP}(G \cdot e, \mathcal{H} \cdot e)$ as desired.
\end{proof}

Importantly, the upper bound for $P_{DP}^*(G,m)$ that Theorem~\ref{thm: deletioncontraction} yields in the case $e_G(u,v)=1$ does not hold when $e_G(u,v) \geq 2$.  To see why, suppose that $G=C_2$ and $e \in E(G)$.  Then, by Propositions~\ref{pro: underlying} and~\ref{pro: fattree}, $P_{DP}^*(G,m) = m(m-1)$, $P_{DP}^*(G-e,m) = m(m-1)$, and $P_{DP}(G \cdot e,m) = m$.  So, it is clear that for any $m \in \N$,  $P^*_{DP}(G,m) > P^*_{DP}(G-e,m) - P_{DP}(G \cdot e, m )$.  

It is easy to see from Corollary~\ref{cor: tree} that the lower bound on the DP color function and upper bound on the dual DP color function provided by Theorem~\ref{thm: deletioncontraction} are tight when $G$ is a tree.  We will now show that these bounds are also tight when $G$ is a cycle.  We begin with a definition from~\cite{KM21}.  Suppose $G=C_n$ with $n \geq 3$, and suppose that the vertices of $G$ in cyclic order are $v_1, v_2, \ldots, v_n$.  An $m$-fold cover $\mathcal{H}=(L,H,M)$ is called a \emph{$C_n$-twister}, if it is possible to name the vertices of $H$ so that: $L(v_i) = \{(v_i,l): l \in [m] \}$, $(v_i,l)$ and $(v_{i+1},l)$ are adjacent in $H$ for each $l \in [m]$ and $i \in [n-1]$, $(v_1,1)$ and $(v_n,m)$ are adjacent in $H$, and $(v_n,l)$ and $(v_1,l+1)$ are adjacent in $H$ for each $l \in [m-1]$.

Suppose $k, m \in \N$ and $m \geq 2$.  It is easy to check that if $G_1=C_{2k+2}$ and $\HH_1$ is an $m$-fold $C_{2k+2}$-twister of $G_1$, then $P_{DP}(G_1, \HH_1) = (m-1)^{2k+2} - 1$.  Similarly, if $G_2=C_{2k+1}$ and $\HH_2$ is an $m$-fold $C_{2k+1}$-twister of $G_2$, then $P_{DP}(G_2, \HH_2) = (m-1)^{2k+1} + 1$.  

\begin{pro} \label{pro: cycle}
Suppose $n,m \geq 2$.  If $n$ is even, then $P_{DP}(C_n,m) = (m-1)^n - 1$ and $P_{DP}^*(C_n,m) = P(C_n,m)$.  If $n$ is odd, $P_{DP}(C_n,m) = P(C_n,m)$ and $P_{DP}^*(C_n,m) = (m-1)^n + 1$. 
\end{pro}

\begin{proof}
The proof is by induction on $n$.  The desired result for $n=2$ follows immediately from Propositions~\ref{pro: underlying} and~\ref{pro: fattree}.  So, suppose that $m \geq 2$, $n \geq 3$, and the desired result holds for all natural numbers greater than 1 and less than $n$.  Suppose that $G = C_n$ and $e \in E(G)$. Notice that $G-e = P_n$ and $G \cdot e = C_{n-1}$.  So, by Corollary~\ref{cor: tree} and Theorem~\ref{thm: deletioncontraction} we have that
$$ m(m-1)^{n-1} - P^*_{DP}(C_{n-1}, m ) \leq P_{DP}(G,m) \leq P^*_{DP}(G,m) \leq m(m-1)^{n-1} - P_{DP}(C_{n-1}, m ).$$
Now, suppose $n$ is odd.  Then, the inductive hypothesis tells us that $P^*_{DP}(C_{n-1}, m) = (m-1)^{n-1} + (m-1)$ and $P_{DP}(C_{n-1}, m) = (m-1)^{n-1} - 1$.  So, we have that 
$$(m-1)^n - (m-1) \leq P_{DP}(G,m) \leq P^*_{DP}(G,m) \leq (m-1)^n + 1.$$
The desired result when $n$ is odd now follows from the fact that  $P_{DP}(G,m) \leq P(C_n,m)$ and $P^*_{DP}(G,m) \geq P_{DP}(G, \HH) = (m-1)^n + 1$ where $\HH$ is an $m$-fold $C_n$-twister of $G$.

Now, suppose that $n$ is even. Then, the inductive hypothesis tells us that $P^*_{DP}(C_{n-1}, m) = (m-1)^{n-1} + 1$ and $P_{DP}(C_{n-1}, m) = (m-1)^{n-1} - (m-1)$.  So, we have that 
$$(m-1)^n - 1 \leq P_{DP}(G,m) \leq P^*_{DP}(G,m) \leq (m-1)^n + (m-1).$$
The desired result when $n$ is even now follows from the fact that  $P^*_{DP}(G,m) \geq P(C_n,m)$ and $P_{DP}(G,m) \leq P_{DP}(G, \HH) = (m-1)^n - 1$ where $\HH$ is an $m$-fold $C_n$-twister of $G$. 
\end{proof}

Now, recall that a unicyclic graph is a connected graph that contains exactly one cycle.  Suppose $G$ is a unicyclic graph on $n$ vertices that contains a $C_l$.  Notice that $G$ can be constructed by starting with the $C_l$ that it contains and then successively adding $n-l$ pendant edges of degree 1.  This fact along with Propositions~\ref{pro: pendant} and~\ref{pro: cycle} immediately yield the following result.

\begin{cor} \label{cor: unicyclic}
Suppose $n, k, m \in \N$ with $m \geq 2$.  If $G$ is a unicyclic graph on $n$ vertices containing a $C_{2k}$, then $P_{DP}(G,m) = (m-1)^n - (m-1)^{n-2k}$ and $P_{DP}^*(G,m) = P(G,m) = (m-1)^n + (m-1)^{n-2k+1}$.  If $G$ is a unicyclic graph on $n$ vertices containing a $C_{2k+1}$ and $m \geq 2$, $P_{DP}(G,m) = P(G,m) = (m-1)^n - (m-1)^{n-2k}$ and $P_{DP}^*(G,m) = (m-1)^n + (m-1)^{n-2k-1}$. 
\end{cor}

It is easy to see that if $m \geq 2$, $G$ is a simple unicyclic graph, and $e$ is an edge of the cycle contained in $G$, then Corollaries~\ref{cor: tree} and~\ref{cor: unicyclic} imply $P_{DP}(G,m) =P_{DP}(G-e,m) - P^*_{DP}(G \cdot e, m )$ and $P^*_{DP}(G,m) = P^*_{DP}(G-e,m) - P_{DP}(G \cdot e, m ).$  Consequently, the lower bound on the DP color function and upper bound on the dual DP color function provided by Theorem~\ref{thm: deletioncontraction} are tight when $G$ is a simple unicyclic graph.~\footnote{In fact, the lower bound on the DP color function in Theorem~\ref{thm: deletioncontraction} is also tight when $G$ is a unicyclic graph containing a $C_2$ and $e$ is an edge of the $C_2$ in $G$.}  

With Questions~\ref{ques: lower} and~\ref{ques: upper} in mind, we are now ready to show that the lower bound on the DP color function and upper bound on the dual DP color function provided by Theorem~\ref{thm: deletioncontraction} may be tight when $G$ is neither a tree nor a simple unicyclic graph.  We begin by showing a family of graphs for which the lower bound on the DP color function is tight.

\begin{pro} \label{pro: fateveng}
Suppose that $G$ is a multigraph with underlying graph $U = C_{2k+2}$.  Suppose that the vertices of $U$ in cyclic order are $v_1, \ldots, v_{2k+2}$.  Also, suppose $e_G(v_i,v_{i+1})=1$ for each $i \in [2k+1]$, $e_G(v_1, v_{2k+2})=l$, and $E_G(v_1, v_{2k+2}) = \{e_1, \ldots, e_l\}$ for some $l \in \N$.  Then, $P_{DP}(G,m)= (m-1)^{2k+1} (m-l) - l$ whenever $m \geq l+1$.
\end{pro}

Before we begin the proof, one should note that this Proposition implies $P_{DP}(G,m) = P_{DP}(G-e_1,m) - P_{DP}^* (G \cdot e_1, m)$ whenever $m \geq l+1$.

\begin{proof}
We will first prove that $P_{DP}(G,m) \leq (m-1)^{2k+1} (m-l) - l$ whenever $m \geq l+1$. Let $G' = G - E_G(v_1, v_{2k+2})$ (note that $G'=P_{2k+2}$).  For any $m \geq l+1$, suppose $\HH' = (L,H',M')$ is an $m$-fold cover of $G'$ with a canonical labeling.  For each $i \in [l]$, let $A_i = \{ (v_1,j)(v_{2k+2},((j+i-1) \text{ mod } m)+1) : j \in [m] \}$.  Let $H$ be the simple graph obtained from $H'$ by adding the edges in $\bigcup_{i=1}^l A_i$.  Let $M$ be the function on $E(G)$ that agrees with $M'$ on $E(G')$ and maps $e_i$ to $A_i$ for each $i \in [l]$.  Now, $\HH = (L,H,M)$ is a full $m$-fold cover of $G$.  We wish to compute $P_{DP}(G, \HH)$.  For each $(i,j) \in [m]^2$, let $C_{(i,j)}$ be the set of proper $m$-colorings of $G$ that color $v_1$ with $i$ and $v_{2k+2}$ with $j$.  By construction, if $P= \{(i,j) \in [m]^2 : ((j-i) \text{ mod } m) \in [l] \}$,
$$P_{DP}(G, \HH) = P_{DP}(G', \HH') - \sum_{(i,j) \in P} |C_{(i,j)}| = P(G',m) - \sum_{(i,j) \in P} |C_{(i,j)}|.$$
Since $m \geq l+1$, we have that if $(i,j) \in P$, then $i \neq j$.  Consequently, $\sum_{(i,j) \in P} |C_{(i,j)}| = lm P(C_{2k+2},m)/(m(m-1)) = l(m-1)^{2k+1} + l$.  So, 
$$P_{DP}(G,m) \leq P_{DP}(G,\HH) = m(m-1)^{2k+1} - \left ( l(m-1)^{2k+1} + l\right) = (m-1)^{2k+1} (m-l) - l.$$

We will now prove that $P_{DP}(G,m)= (m-1)^{2k+1} (m-l) - l$ whenever $m \geq l+1$ by induction on $l$.  Notice that when $l=1$ we have the desired result by Proposition~\ref{pro: cycle}.  So, suppose that $l \geq 2$ and the desired result holds for all natural numbers less than $l$.  By the induction hypothesis, for any $m \geq l+1$, $P_{DP}(G-e_1,m) = (m-1)^{2k+1} (m-l+1) - l+1$.  This along with Theorem~\ref{thm: deletioncontraction} and Proposition~\ref{pro: cycle} yields
\begin{align*}
P_{DP}(G,m) &\geq P_{DP}(G-e_1,m) - P_{DP}^* (G \cdot e_1, m) \\
 &=(m-1)^{2k+1} (m-l+1) - l+1 - P_{DP}^*(C_{2k+1},m) \\
&= (m-1)^{2k+1} (m-l+1) - l+1 - \left( (m-1)^{2k+1} + 1 \right) = (m-1)^{2k+1} (m-l) - l.
\end{align*} 
This lower bound along with the upper bound established in the first part of this proof completes the induction step.
\end{proof}

Finally, we provide an example of a family of graphs for which the upper bound on the dual DP color function in Theorem~\ref{thm: deletioncontraction} is tight.  For this example, we need a result~\footnote{The result in~\cite{BH21} is only for simple graphs, but it is easy to extend it to 2-cycles by Proposition~\ref{pro: pendant}.} from~\cite{BH21}.

\begin{pro} [\cite{BH21}] \label{pro: glue}
Suppose $k,l \in \N$, $G_1 = C_{2k}$, $G_2 = C_{2l}$, and $|V(G_1) \cap V(G_2)|=1$.  If $G$ is the graph with vertex set $V(G_1) \cup V(G_2)$ and edge set $E(G_1) \cup E(G_2)$, then
$$P_{DP}(G,m) = \frac{P_{DP}(G_1,m)P_{DP}(G_2,m)}{m} = \frac{((m-1)^{2k}-1)((m-1)^{2l}-1)}{m}$$
whenever $m \geq 2$.
\end{pro}

\begin{pro} \label{pro: twoodd}
Suppose $k,l \in \N$, $G_1 = C_{2k+1}$, $G_2 = C_{2l+1}$, $|V(G_1) \cap V(G_2)|=2$, and $E(G_1) \cap E(G_2)= \{uv\}$.  If $G$ is the graph with vertex set $V(G_1) \cup V(G_2)$ and edge set $E(G_1) \cup E(G_2)$, then for each $m \geq 2$, $P_{DP}^*(G,m) = (m-1)^{2k+2l} + (m-1) - ((m-1)^{2k}-1)((m-1)^{2l}-1)/m$.  Consequently, $P_{DP}^*(G,m) = P_{DP}^* (G - uv,m) - P_{DP}(G \cdot uv,m)$
\end{pro}

\begin{proof}
Let $G' = G - uv$ (note that $G'=C_{2k+2l}$).  For any $m \geq 2$, suppose $\HH' = (L,H',M')$ is an $m$-fold cover of $G'$ with a canonical labeling.  Let $H$ be the graph obtained from $H'$ by adding the edges $(u,j)(v,j+1)$ for each $j \in [m-1]$ and the edge $(u,m)(v,1)$.  Let $M$ be the function on $E(G)$ that agrees with $M'$ on $E(G')$ and maps $uv$ to $E_{H}(L(u),L(v))$.  Now, $\HH = (L,H,M)$ is a full $m$-fold cover of $G$.  We wish to compute $P_{DP}(G, \HH)$.  For each $(i,j) \in [m]^2$, let $C_{(i,j)}$ be the set of proper $m$-colorings of $G$ that color $u$ with $i$ and $v$ with $j$.  By construction,
\begin{align*}
P_{DP}(G, \HH) &= P_{DP}(G', \HH') - |C_{(m,1)}| - \sum_{j=1}^{m-1} |C_{(j,j+1)}| \\
&= P(C_{2k+2l},m) - m \left ( \frac{P(C_{2k+1},m)P(C_{2l+1},m)}{m^2(m-1)^2} \right) \\
&= (m-1)^{2k+2l} + (m-1) - \frac{((m-1)^{2k}-1)((m-1)^{2l}-1)}{m}.
\end{align*}
So, $(m-1)^{2k+2l} + (m-1) - ((m-1)^{2k}-1)((m-1)^{2l}-1)/m \leq P_{DP}^*(G,m)$.  Now, by Theorem~\ref{thm: deletioncontraction} we know that for each $m \geq 2$, $P_{DP}^*(G,m) \leq P_{DP}^* (G - uv,m) - P_{DP}(G \cdot uv, m)$.  By Proposition~\ref{pro: cycle}, we know that $P_{DP}^* (G - uv,m) = P(C_{2k+2l},m) = (m-1)^{2k+2l} + (m-1)$.  By Proposition~\ref{pro: glue}, we know that $P_{DP}(G \cdot uv,m) = ((m-1)^{2k}-1)((m-1)^{2l}-1)/m$.  This completes the proof. 
\end{proof} 

\section{Asymptotics of the Dual DP Color Function} \label{dual}

In this Section we wish to prove Theorem~\ref{thm: dual} which we restate.

\begin{customthm} {\ref{thm: dual}}
Suppose $g$ is an odd integer with $g \geq 3$.  If $G$ is a simple graph on $n$ vertices with girth at least $g-1$, then $P^*_{DP}(G, m) - P(G,m) = O(m^{n-g+1})$ as $m \rightarrow \infty$. 
\end{customthm}

Our main tools for proving this Theorem are: Theorem~\ref{thm: general}, our deletion-contraction relation, and the following generalization of Theorem~\ref{thm: general} to multigraphs.  Theorem~\ref{thm: easy} is an essential tool in our proof of Theorem~\ref{thm: dual} since contracting an edge in a simple graph of girth 3 can produce a multigraph, and Theorem~\ref{thm: general} does not apply to multigraphs.

\begin{thm} \label{thm: easy}
For any multigraph $G$ on $n$ vertices, $P(G,m) - P_{DP}(G,m) = O(m^{n-1})$ as $m \rightarrow \infty$.
\end{thm}

\begin{proof}
When $G$ has no edges the result is obvious since $P(G,m) = P_{DP}(G,m)$ for all $m \in \N$ when $G$ is edgeless.  So, suppose that $G$ has at least one edge and $v_1, \ldots, v_n$ is an ordering of the elements of $V(G)$ such that there are precisely $d_i$ edges with $v_i$ as an endpoint and some vertex preceding $v_i$ in the ordering as the other endpoint for each $i \in [n]$. Let $D = \max_{i \in [n]} d_i$.  Now, by Proposition~\ref{pro: order}, we know that for any $m \geq D$
$$ P(G,m) - P_{DP}(G,m) \leq P(G,m) - \prod_{i=1}^n (m-d_i).$$
Since it is well-known (see e.g.,~{W01}) that the leading term of $P(G,m)$ is $m^n$, and the leading term of $\prod_{i=1}^n (m-d_i)$ is $m^n$, the result follows. 
\end{proof}

We are now ready to present our proof of Theorem~\ref{thm: dual}.

\begin{proof}
Suppose $g$ is a fixed odd integer with $g \geq 3$.  We will prove the desired result by induction on $|E(G)|$.  Suppose that $G$ is a simple graph on $n$ vertices with girth at least $g-1$.  The result is clear when $|E(G)| \in [g-2]$ since $G$ is a tree in this case and $P^*_{DP}(G, m) = P(G,m)$ for all $m \in \N$ by Corollary~\ref{cor: tree}.  So, suppose that $|E(G)| \geq g-1$ and the result holds for natural numbers less than $|E(G)|$.  We may suppose that $G$ contains a cycle since the result is clear otherwise.  Suppose $e$ is an edge in $G$ that is contained in a cycle.  Notice that in the case $g \geq 5$, $G \cdot e$ is a simple graph on $n-1$ vertices with girth at least $g-2$.  In the case that $g=3$, $G \cdot e$ is a multigraph on $n-1$ vertices.  Also, $G - e$ is a graph on $n$ vertices with girth at least $g-1$.  By the Deletion-Contraction Formula and Theorem~\ref{thm: deletioncontraction}, we see that for any $m \in \N$
\begin{align*}
P^*_{DP}(G,m) - P(G,m) &\leq P^*_{DP}(G-e,m) - P_{DP}(G \cdot e, m ) - P(G,m) \\
&= P^*_{DP}(G-e,m) - P(G-e,m) +  P(G \cdot e, m ) - P_{DP}(G \cdot e, m ).
\end{align*}
By the inductive hypothesis and Theorem~\ref{thm: general} when $g \geq 5$ (resp. Theorem~\ref{thm: easy} when $g=3$), we know there are $N_1, N_2 \in \N$ along with constants $C_1$ and $C_2$ such that $P^*_{DP}(G-e,m) - P(G-e,m) \leq C_1 m^{n-g+1}$ whenever $m \geq N_1$ and  $P(G \cdot e, m ) - P_{DP}(G \cdot e, m ) \leq C_2 m ^{n-1 - (g-2)}$ whenever $m \geq N_2$.  So, if we let $N = \max \{N_1,N_2\}$, we have that
$$P^*_{DP}(G,m) - P(G,m) \leq (C_1 + C_2) m^{n-g+1}$$
whenever $m \geq N$.  It follows that $P^*_{DP}(G, m) - P(G,m) = O(m^{n-g+1})$ as $m \rightarrow \infty$.
\end{proof}

By Proposition~\ref{pro: cycle} we know that for each $k \in \N$ $P^*_{DP}(C_{2k+1}, m) - P(C_{2k+1},m) = \Theta(m)$ as $m \rightarrow \infty$ which demonstrates the tightness of Theorem~\ref{thm: dual} for all possible $g$.
\\

{\bf Acknowledgment.}  The author would like to thank Hemanshu Kaul for many helpful conversations regarding the results of this paper.

\end{document}